\documentclass[11pt,draftcls,onecolumn]{IEEEtran}

\usepackage[latin1]{inputenc}
\usepackage{amsmath,amssymb}
\usepackage{graphicx}
\usepackage{amsthm}
\usepackage{cite}
\usepackage{array}

 \newtheoremstyle{IEEE}
 {0pt}
 {0pt}
 {\normalfont}
 {12pt}
 {\normalfont\itshape}
 {:}
 { }
 {\thmname{#1\,}\thmnumber{#2}\thmnote{#3}}
 \theoremstyle{IEEE}

\newtheorem{theorem}{Theorem}
\newtheorem{lemma}{Lemma}
\newtheorem{remark}{Remark}

\newtheorem{corollary}{Corollary}
\newtheorem{example}{Example}

\allowdisplaybreaks
\date{}

\begin{document}

\title{New summation inequalities and their applications to discrete-time delay systems}

\author{Le Van Hien and Hieu Trinh\thanks{{\em}}

\thanks{L.V. Hien is with the School of Engineering,
Deakin University, VIC 3217, Australia, and the Department of Mathematics,
Hanoi National University of Education, Hanoi, Vietnam (e-mail: hienlv@hnue.edu.vn).}
\thanks{H. Trinh is with the School of Engineering,
Deakin University, VIC 3217, Australia (e-mail: hieu.trinh@deakin.edu.au).}}

\maketitle
\pagestyle{plain}

\begin{abstract}
This paper provides new summation inequalities in both single and double forms
to be used in stability analysis of discrete-time systems with time-varying delays.
The potential capability of the newly derived inequalities is
demonstrated by establishing less conservative
stability conditions for a class of linear discrete-time systems
with an interval time-varying delay in the framework of linear matrix inequalities.
The effectiveness and least conservativeness of the derived stability conditions
are shown by academic and practical examples.
\end{abstract}


\section{Introduction}\label{sec:1}
It is well known that time-delay frequently occurs in practical systems
and usually is a source of bad performance, oscillations or instability
\cite{SNAMG,F1}. Therefore, the problem of stability analysis
and applications to control of time-delay systems are essential and
of great importance for both theoretical and practical reasons
which have attracted considerable attention, see, for example
\cite{FSL,HP09,F2,Seu12,SG1,HH14,LGZX,PLL} and the references therein.

Among existing works which concern with stability
of linear time-delay systems, the Lyapunov-Krasovskii functional (LKF)
method plays an essential role in deriving efficient stability conditions. Based on a priori construction
of a Lyapunov-Krasovskii functional combining with some bounding techniques
\cite{HP09,SG1,PLL,PKJ,FLY}, improved delay-dependent stability conditions
for continuous/discrete time-delay systems were derived in terms of tractable
linear matrix inequalities \cite{SG1,PLL,FLY,Peng}.
However, the design of such Lyapunov-Krasovskii functional
and especially the techniques used in bounding the derivative or difference
of constructed Lyapunov-Krasovskii functional usually produce an undesirable
conservatism in stability conditions. Therefore, aiming at reducing conservativeness
of stability conditions, it is relevant and important to improve some fundamental inequalities
to be used in establishing such stability criteria \cite{PLL,SG2}.

Note that, most of the aforementioned works have been devoted to continuous-time systems.
Besides, the problem of stability analysis and control of
discrete-time systems with time-varying delay is very relevant and
therefore it should be receiving a greater focus due to the following practical reasons.
Firstly, with the rapid development of computer-based computational techniques,
discrete-time systems are more suitable for computer simulation, experiment and computation.
Secondly, many practical systems are in the form of 
nonlinear and/or non-autonomous continuous-time systems 
with time-varying delays. A discretization from continuous-time systems
leads to discrete-time systems described by difference equations which inherit the 
similar dynamical behavior of the continuous ones\cite{H14}. 
In addition, the investigation of stability and control
of discrete-time systems requires specific and quite different tools from the continuous ones.
Thus, stability analysis and control of discrete-time delay systems have received considerable
attention in recent years \cite{F3,HWLX,ZXZ,MLDG,SH,Kao,Kwon1,HAT,Kim}.
Most recently, novel summation inequalities were derived \cite{NPH,SGF1}
by extending the Wirtinger-based integral inequality \cite{SG1}.
These summation inequalities provide a powerful tool to derive less conservative
stability conditions for discrete-time systems with interval time-varying delay
in the framework of tractable linear matrix inequalities.

In this paper, new summation inequalities which provide an efficient tool
for stability analysis of discrete-time systems with time-varying delay are fisrt derived.
Inspired by the approaches proposed in \cite{PLL,SG2} for
the continuous-time systems, new summation inequalities in both single and double
forms are derived by refining the discrete Jensen inequalities.
It is worth noting that the obtained results in this paper theoretically encompass
the summation inequalities proposed in \cite{SGF1,NPH}. Furthermore,
unlike \cite{SGF1,NPH}, we prove that the proposed inequalities
do not depend on the choice of first-order approximation sequences.
By employing these new inequalities, a suitable Lyapunov-Krasovskii functional
is constructed and less conservative stability conditions are derived
for a class of discrete-time systems with interval time-varying delay.
To illustrate the effectiveness of the proposed stability conditions,
an academic example and a practical satellite control system
are provided. These examples show that our stability conditions provide
significant improvement over existing works in the literature.

The  paper is organized as follows. Section II
presents some preliminary results. New summation inequalities and their applications to stability analysis of
a class of discrete-time systems with interval time-varying delay
are presented in Section III and Section IV, respectively.
Numerical examples to demonstrate the effectiveness of the obtained results
are also given in Section IV.

\section{Preliminaries}

\textbf{Notations:} Throughout this paper, we denote $\Bbb{Z}$ and $\Bbb{Z}^+$ the set of integers
and positive integers, respectively, $\Bbb{R}^n$ the $n$-dimensional Euclidean space with vector norm
$\|.\|$, $\Bbb{R}^{n\times m}$ the set of $n\times m$ real matrices. For matrices
$A,B\in\Bbb{R}^{n\times m}$, $\mathrm{col}\{A,B\}$ and $\mathrm{diag}\{A,B\}$
denote the block matrix $\begin{bmatrix}A\\ B\end{bmatrix}$ and
$\begin{bmatrix}A&0\\0&B\end{bmatrix}$, respectively. A matrix $P\in\Bbb{R}^{n\times n}$
is positive (negative) definite, write $P>0$ ($P<0$) if $x^TPx>0$ ($x^TPx<0$) for all $x\in\Bbb{R}^n$, $x\neq 0$.
We let $\Bbb{S}^+_n$ denote the set of symmetric positive definite matrices.
For any $A\in\Bbb{R}^{n\times n}$, $\mathtt{He}(A)$ stands for $A+A^T$.
For $a,b\in\Bbb{Z}$, $a\leq b$, $\Bbb{Z}[a,b]$ denotes the set of integers between $a$ and $b$.
For a sequence $u:\Bbb{Z}[a,b]\to\Bbb{R}^n$,
we write $u_k=u(k)$, $k\in\Bbb{Z}[a,b]$, $\Delta$ denotes the forward difference operator,
that means $\Delta u_k=u_{k+1}-u_k$. For any two sequences $u,v:\Bbb{Z}[a,b]\to\Bbb{R}^n$,
it is obvious that $u_k\Delta v_k=\Delta(u_kv_k)-v_{k+1}\Delta u_k$.

The following inequalities which are widely used in the literature
can be easily derived by using Schur complement lemma.

\begin{lemma}(Jensen's inequalities)\label{lm2.1}
For a given matrix $R\in\Bbb{S}^+_n$, integers $b>a$, any sequence
$u: \Bbb{Z}[a,b]\to\Bbb{R}^n$, the following inequalities hold
\begin{align}
&\sum_{k=a}^bu^T_kRu_k\geq\frac{1}{\ell}\big(\sum_{k=a}^bu_k\big)^T
R\big(\sum_{k=a}^bu_k\big),\label{e2.1}\\
&\sum_{k=a}^b\sum_{s=a}^ku_s^TRu_s\geq\frac{2}{\ell(\ell+1)}\big(\sum_{k=a}^b\sum_{s=a}^ku_s\big)^T
R\big(\sum_{k=a}^b\sum_{s=a}^ku_s\big),\label{e2.2}
\end{align}
where $\ell=b-a+1$ denotes the length of interval $[a,b]$ in $\Bbb{Z}$.
\end{lemma}

\section{New summation inequalities}

In this section, new summation inequalities are derived by refining \eqref{e2.1}, \eqref{e2.2}. 
In the following, let us denote $J_1(u)$ and $J_2(u)$ as the gap of \eqref{e2.1} and 
\eqref{e2.2}, respectively, that is the difference between the
left-hand side and the right-hand side in \eqref{e2.1} and \eqref{e2.2}.
By refining \eqref{e2.1} and \eqref{e2.2}, we aim to find new lower bounds for $J_1(u), J_2(u)$
other than zero. 

\begin{lemma}\label{lm3.1}
For a given matrix $R\in\Bbb{S}^+_n$, integers $b>a$,
any sequence $u: \Bbb{Z}[a,b]\to\Bbb{R}^n$, the following inequality holds
\begin{equation}\label{e3.1}
J_1(u)\geq\frac{3(\ell+1)}{\ell(\ell-1)}\zeta_1^TR\zeta_1
+\frac{5(\ell+1)(\ell+2)^2}{\ell(\ell-1)(\ell^2+11)}\zeta_2^TR\zeta_2
\end{equation}
where $\zeta_1=\upsilon_1-\frac{2}{\ell+1}\upsilon_2$,
$\zeta_2=\upsilon_1-\frac{6}{\ell+1}\upsilon_2+\frac{12}{(\ell+1)(\ell+2)}\upsilon_3$
and $\upsilon_1=\sum\limits_{k=a}^bu_k$,
$\upsilon_2=\sum\limits_{k=a}^b\sum\limits_{s=a}^{k}u_s$,
$\upsilon_3=\sum\limits_{k=a}^b\sum\limits_{s=a}^k\sum\limits_{i=a}^su_i$.
\end{lemma}

\begin{proof} Note at first that if $\ell=1$ then $\zeta_1=\zeta_2=0$ and thus \eqref{e3.1}
holds. Now assume that $\ell>1$. We use the idea of bilevel optimization
to get \eqref{e3.1} by refining \eqref{e2.1}. To this, for a sequence $u: \Bbb{Z}[a,b]\to \Bbb{R}^n$,
we define an approximation sequence $v: \Bbb{Z}[a,b]\to\Bbb{R}^n$ as follows
\begin{equation}\label{e3.2}
v_k=u_k-\frac{1}{\ell}\sum_{k=a}^bu_k+\alpha_k\chi_1+\beta_k\chi_2
\end{equation}
where $\alpha_k$ and $\beta_k$ are two sequences of real numbers
and $\chi_1,\chi_2\in\mathbb{R}^n$ are constant vectors which will be defined later.
From \eqref{e3.2} we have
\begin{equation}\label{e3.3}
\begin{aligned}
\sum_{k=a}^bv_k^TRv_k=\; &J_1(u)+2\chi_1^TR(\sum_{k=a}^b\alpha_ku_k)
+2\chi_2^TR(\sum_{k=a}^b\beta_ku_k)\\
&-\frac{2}{\ell}(\sum_{k=a}^b\alpha_k)\chi_1^TR\upsilon_1
-\frac{2}{\ell}(\sum_{k=a}^b\beta_k)\chi_2^TR\upsilon_1\\
&+(\sum_{k=a}^b\alpha_k^2)\chi_1^TR\chi_1+(\sum_{k=a}^b\beta_k^2)\chi_2^TR\chi_2\\
&+2(\sum_{k=a}^b\alpha_k\beta_k)\chi_1^TR\chi_2.
\end{aligned}
\end{equation}

Let $\hat{u}_k=\sum\limits_{i=a}^{k-1}u_i$ for $k>a$, $\hat{u}_k=0$ for $k=a$,
then $u_k=\Delta\hat{u}_k$ and, consequently,
$\alpha_ku_k=\Delta(\alpha_k\hat{u}_k)-\hat{u}_{k+1}\Delta\alpha_k$.
Taking summation from $a$ to $b$ gives
\begin{equation}\label{e3.4}
\sum_{k=a}^b\alpha_ku_k=\alpha_{b+1}\sum_{k=a}^bu_k
-\sum_{k=a}^b\hat{u}_{k+1}\Delta\alpha_k.
\end{equation}

For any first-order sequence $\alpha_k$
which can be written as $\alpha_k=c_0(k-a)+c_1$, $c_0\neq 0$, we have
\[
\alpha_{b+1}=c_0\ell+c_1,\; \Delta\alpha_k=c_0,\;
\sum_{k=a}^b\alpha_k=c_0\frac{\ell(\ell-1)}{2}+c_1\ell.
\]
This, in regard to \eqref{e3.4}, leads to
\begin{equation}\label{e3.5}
2\chi_1^TR(\sum_{k=a}^b\alpha_ku_k)
-\frac{2}{\ell}(\sum_{k=a}^b\alpha_k)\chi_1^TR\upsilon_1=\alpha_0(\ell+1)\chi_1^TR\zeta_1.
\end{equation}

Similar to \eqref{e3.4} we have
\[
\sum_{k=a}^b\beta_ku_k=\beta_{b+1}\sum_{k=a}^bu_k
-\sum_{k=a}^b\hat{u}_{k+1}\Delta\beta_k.
\]
At this time we define the sequence $\tilde{u}_k=\sum_{s=a}^k\hat{u}_s$
then $\hat{u}_{k+1}=\Delta\tilde{u}_k$ and thus
\[
\hat{u}_{k+1}\Delta\beta_k=\Delta(\Delta\beta_k\tilde{u}_k)-\tilde{u}_{k+1}\Delta^2\beta_k.
\]

For convenience, we choose $\beta_k=(k-a)^2-\ell(k-a)+\frac{\ell^2-1}{6}$
then $\sum_{k=a}^b\beta_k=0$,
$\beta_{b+1}=\frac{\ell^2-1}{6}$, $\Delta\beta_{b+1}=\ell+1$
and $\Delta^2(\beta_k)=2$. Note also that
\[
\sum_{k=a}^b\tilde{u}_{k+1}=\sum_{k=a}^b\sum_{s=a}^{k+1}\hat{u}_i
=\sum_{k=a}^b\sum_{s=a+1}^{k+1}\hat{u}_i
=\sum_{k=a}^b\sum_{s=a}^{k}\hat{u}_{i+1}=\upsilon_3.
\]
Therefore
\begin{equation}\label{e3.6}
2\chi_2^TR(\sum_{k=a}^b\beta_ku_k)
-\frac{2}{\ell}(\sum_{k=a}^b\beta_k)\chi_2^TR\upsilon_1=\frac{1}{3}\chi_2^TR\zeta_3
\end{equation}
where $\zeta_3=(\ell^2-1)\upsilon_1-6(\ell+1)\upsilon_2+12\upsilon_3$.

On the other hand, from \eqref{e3.2} and note that $\sum_{k=a}^b\beta_k=0$, we readily obtain
$\sum_{k=a}^bv_k=(\sum_{k=a}^b\alpha_k)\chi_1$. This, together with \eqref{e3.3},
\eqref{e3.5} and \eqref{e3.6}, leads to
\begin{equation}\label{e3.7}
\begin{aligned}
J_1(v)=\; &J_1(u)+c_0(\ell+1)\chi_1^TR\zeta_1+\frac{1}{3}\chi_2^TR\zeta_3\\
&+\Big[\sum_{k=a}^b\alpha_k^2-\frac{1}{\ell}\big(\sum_{k=a}^b\alpha_k\big)^2\Big]\chi_1^TR\chi_1\\
&+(\sum_{k=a}^b\beta_k^2)\chi_2^TR\chi_2+2(\sum_{k=a}^b\alpha_k\beta_k)\chi_1^TR\chi_2.
\end{aligned}
\end{equation}
It can be verified by some direct computations that
\begin{align*}
&\sum_{k=a}^b\alpha_k^2-\frac{1}{\ell}\big(\sum_{k=a}^b\alpha_k\big)^2=c_0^2\frac{\ell(\ell^2-1)}{12},\\
&\sum_{k=a}^b\beta_k^2=\frac{\ell(\ell^2-1)(\ell^2+11)}{180},\;
\sum_{k=a}^b\alpha_k\beta_k=-c_0\frac{\ell(\ell^2-1)}{12}.
\end{align*}
By injecting those equalities into \eqref{e3.7} we then obtain
\begin{equation}\label{e3.8}
J_1(v)=J_1(u)+c_0(\ell+1)\chi_1^TR\zeta_1+\frac{c_0^2\ell(\ell^2-1)}{12}\chi_1^TR\chi_1+\hat{J}
\end{equation}
where
$\hat{J}=\frac{1}{3}\chi_2^TR\zeta_3-c_0\frac{\ell(\ell^2-1)}{6}\chi_1^TR\chi_2
+\frac{\ell(\ell^2-1)(\ell^2+11)}{180}\chi_2^TR\chi_2.$

Now, at the first stage we define $\chi_1=\frac{-\lambda}{c_0}\zeta_1$, where $\lambda$ is a scalar, then
by Lemma \ref{lm2.1}, $J_1(v)\geq 0$, it follows from \eqref{e3.8} that
\begin{equation}\label{e3.9}
J_1(u)\geq (\ell+1)\Big(\lambda-\frac{\ell(\ell-1)}{12}\lambda^2\Big)\zeta_1^TR\zeta_1-\hat{J}.
\end{equation}
The function $\lambda-\frac{\ell(\ell-1)}{12}\lambda^2$
attains its maximum $\frac{3}{\ell(\ell-1)}$
at $\lambda=\frac{6}{\ell(\ell-1)}$, and hence $\chi_1=\frac{-6}{c_0\ell(\ell-1)}\zeta_1$,
then from \eqref{e3.9} we obtain
$J_1(u)\geq \frac{3(\ell+1)}{\ell(\ell-1)}\zeta_1^TR\zeta_1-\hat{J}$.
In addition, by injecting $\chi_1=\frac{-6}{c_0\ell(\ell-1)}\zeta_1$ into $\hat{J}$
we then obtain
\begin{equation}\label{e3.10}
\begin{aligned}
J_1(u)\geq \frac{3(\ell+1)}{\ell(\ell-1)}\zeta_1^TR\zeta_1
&-\frac{\ell(\ell^2-1)(\ell^2+11)}{180}\chi_2^TR\chi_2\\
&-\frac{(\ell+1)(\ell+2)}{3}\chi_2^TR\zeta_2.
\end{aligned}
\end{equation}
As this stage, we define $\chi_2=-3\theta\zeta_2$,
$\theta$ is a scalar, then by some similar lines
when dealing with \eqref{e3.9} we finally obtain
\eqref{e3.1} which completes the proof.
\end{proof}

\begin{remark}
The proof of Lemma \ref{lm3.1} can be shortened
by a specific selection of $\alpha_k$, for example,
$\alpha_k=(k-a)-\frac{\ell(\ell-1)}{2}$.
\end{remark}

\begin{remark}\label{rm3.1}
Lemma \ref{lm3.1} in this paper generalizes the summation inequality
derived in Lemma 2 in \cite{SGF1} and Lemma 3 in \cite{NPH} by the following points.
Firstly, the inequality provided in Lemma \ref{lm3.1} in this paper
encompasses both the inequalities proposed in Lemma 2 in \cite{SGF1} and Lemma 3 in \cite{NPH}
since a positive term is added into the right-hand side of \eqref{e3.1}.
Secondly, and most interesting is that, \eqref{e3.1} can be derived from
the approximation \eqref{e3.2} for any first-order sequence
$\alpha_k=c_0k+c_1, c_0\neq 0$ whereas some special cases of \eqref{e3.2}
were used to derive Lemma 2 in \cite{SGF1} and Lemma 3 in \cite{NPH}.
Thirdly, a unify approach is introduced to derive some new lower bounds of summation estimate
in both single and double form proposed in Lemma \ref{lm3.1} and the following lemmas. 
\end{remark}

\begin{lemma}\label{lm3.2}
For a given matrix $R\in\Bbb{S}^+_n$, integers $b>a$,
any sequence $u: \Bbb{Z}[a,b]\to\Bbb{R}^n$, the following inequality holds
\begin{equation}\label{e3.11}
J_2(u)\geq\frac{16(\ell+2)}{\ell(\ell^2-1)}\zeta_4^TR\zeta_4
\end{equation}
where $\zeta_4=\upsilon_2-\frac{3}{\ell+2}\upsilon_3$.
\end{lemma}

\begin{proof}
Similar to Lemma \ref{lm3.1}, when $\ell=1$,
$\zeta_4=0$ and \eqref{e3.11} is trivial.
Assume that $\ell>1$. By the same approach used in deriving \eqref{e3.1},
we now construct the following approximation
\begin{equation}\label{e3.12}
v_k=u_k-\frac{2}{\ell(\ell+1)}\sum_{k=a}^b\sum_{s=a}^ku_s+\alpha_k\chi
\end{equation}
for a given sequence $u: \Bbb{Z}[a,b]\to\Bbb{R}^n$. Similar to \eqref{e3.3}
\begin{equation}\label{e3.13}
\begin{split}
&\sum_{k=a}^b\sum_{s=a}^kv_s^TRv_s=J_2(u)+2\chi^TR(\sum_{k=a}^b\sum_{s=a}^k\alpha_su_s)\\
&-\frac{4}{\ell(\ell+1)}(\sum_{k=a}^b\sum_{s=a}^k\alpha_s)\chi^TR(\sum_{k=a}^b\sum_{s=a}^ku_s)
+(\sum_{k=a}^b\sum_{s=a}^k\alpha_s^2)\chi^TR\chi.
\end{split}
\end{equation}

For any first-order sequence $\alpha_k=c_0(k-a)+c_1$, $c_0\neq 0$, by some similar
lines in the proof of Lemma \ref{lm3.1} we have
\begin{equation}\label{e3.14}
J_2(v)=J_2(u)+\frac{c_0^2\ell(\ell+2)(\ell^2-1)}{36}\chi^TR\chi+\frac{4c_0(\ell+2)}{3}\chi^TR\zeta_4.
\end{equation}

From Lemma \ref{lm2.1}, $J_2(v)\geq 0$, and by choosing $\chi=\frac{-3\lambda}{c_0}\zeta_4$,
it follows from \eqref{e3.14} that
\begin{equation}\label{e3.15}
J_2(u)\geq(\ell+2)\Big[4\lambda-\frac{\ell(\ell^2-1)}{4}\lambda^2\Big]\zeta_4^TR\zeta_4
\end{equation}
which yields \eqref{e3.11} for $\lambda=\frac{8}{\ell(\ell^2-1)}$.
The proof is completed.
\end{proof}

\begin{remark}
The summation inequalities proposed in Lemma \ref{lm3.1} and Lemma \ref{lm3.2}
in this paper give a new lower bound for the gap $J_1(u)$ and $J_2(u)$ of the discrete Jensen's inequalities, respectively.
In other words, new refinements of the celebrated Jensen's inequalities have been derived in this paper.
\end{remark}

\begin{remark}
The double summation inequality provided in \eqref{e3.11}
is closely related to the function-based double integral inequality
proposed in \cite{PLL} although the proof \eqref{e3.11} is based on a simple idea,
refining the classical discrete Jensen's inequality. Furthermore, as shown in the proof of Lemma \ref{lm3.2},
inequality \eqref{e3.11} can be derived from \eqref{e3.12} for any first-order sequence $\alpha_k$.
\end{remark}

\begin{remark}
As discussed in \cite{SGF1}, some coefficients in \eqref{e3.1} and \eqref{e3.11} might be difficult to handle, especially
in applications to discrete-time delay systems. Therefore \eqref{e3.1} and \eqref{e3.11} will be reduced
to simpler forms as presented in the following corollary.
\end{remark}

\begin{corollary}(Refined Jensen-based inequalities)\label{lm3.3}
For a given matrix $R\in\Bbb{S}^+_n$, integers $b>a$, any sequence
$u: \Bbb{Z}[a,b]\to\Bbb{R}^n$, the following inequalities hold
\begin{align}
&\sum_{k=a}^bu^T_kRu_k\geq\frac{1}{\ell}\upsilon_1^TR\upsilon_1
+\frac1{\ell}\begin{bmatrix}\zeta_1\\\zeta_2\end{bmatrix}^T
\begin{bmatrix}3R&0\\0&5R\end{bmatrix}
\begin{bmatrix}\zeta_1\\\zeta_2\end{bmatrix},\label{e3.16}\\
&\sum_{k=a}^b\sum_{s=a}^ku_s^TRu_s\geq\frac{2}{\ell(\ell+1)}
\begin{bmatrix}\upsilon_2\\ \zeta_4\end{bmatrix}^T
\begin{bmatrix}R &0\\ 0&8R\end{bmatrix}
\begin{bmatrix}\upsilon_2\\ \zeta_4\end{bmatrix},\label{e3.17}
\end{align}
where $\ell$, $\upsilon_2,\zeta_1,\zeta_2$ and $\zeta_4$ are defined in \eqref{e3.1} and \eqref{e3.11}.
\end{corollary}

\begin{proof}
The proof is straight forward from \eqref{e3.1}, \eqref{e3.11} and thus is omitted here.
\end{proof}

\section{Stability of discrete-time systems with time-varying delay}\label{sec:4}

This section aims to demonstrate the effectiveness of our
newly derived summation inequalities through
applications to stability analysis of 
discrete-time systems with interval time-varying delay.

\subsection{Stability conditions}

Consider a linear discrete-time system with interval time-varying delay of the form
\begin{equation}\label{e4.1}
\begin{cases}
x(k+1)=Ax(k)+A_dx(k-h(k)),\; k\geq 0,\\
x(k)=\phi(k),\; t\in[-h_2,0],
\end{cases}
\end{equation}
where $x(k)\in\Bbb{R}^n$ is the state, $A,A_d\in\mathbb{R}^{n\times n}$ are given matrices,
$h(k)$ is time-varying delay satisfying $h_1\leq h(k)\leq h_2$, where $h_1\leq h_2$ are known
positive integers. For simplicity, hereafter the delay $h(k)$ will be denoted by $h$.

Let $\{e_i^*\}_{1\leq i\leq 10}$ be the row basic of $\Bbb{R}^{10}$
and $e_i=e_i^*\otimes I_n$. We denote $\mathcal{A}=(A-I_n)e_1+A_de_3$ and
\begin{IEEEeqnarray*}{l}
\zeta_0(k)=\mathrm{col}\left\{\begin{bmatrix}x(k)\\ x(k-h_1)\\ x(k-h)\\ x(k-h_2)\end{bmatrix},
\begin{bmatrix}\nu_1(k)\\\nu_2(k) \\ \nu_3(k)\end{bmatrix},
\begin{bmatrix}\nu_4(k)\\ \nu_5(k)\\ \nu_6(k)\end{bmatrix}\right\},\\
\nu_1(k)=\frac1{T(h_1)}\sum_{s=k-h_1}^{k}x(s), 
\nu_2(k)=\frac1{T(h-h_1)}\sum_{s=k-h}^{k-h_1}x(s),\\
\nu_3(k)=\frac1{T(h_2-h)}\sum_{s=k-h_2}^{k-h}x(s), 
\nu_4(k)=\frac1{\gamma(h_1)}\sum_{s=-h_1}^0\sum_{i=k+s}^kx(i),\\
\nu_5(k)=\frac1{\gamma(h-h_1)}\sum_{s=-h}^{-h_1}\sum_{i=k+s}^{k-h_1}x(i), \\
\nu_6(k)=\frac1{\gamma(h_2-h)}\sum_{s=-h_2}^{-h}\sum_{i=k+s}^{k-h}x(i),\\
T(h)=h+1, \gamma(h)=\frac{T(h)T(h+1)}{2},\\
\Omega(h)=\mathrm{col}\left\{e_1, T(h_1)e_5, T(h-h_1)e_6+T(h_2-h)e_7, \gamma(h_1)e_8\right\},\\
\Omega_1=\mathrm{col}\left\{-\mathcal{A}, e_2, e_3+e_4, T(h_1)e_5\right\},\\
\Omega_2=\mathrm{col}\{0,e_1,e_2+e_3,T(h_1)e_1\},\\
 \Gamma_1=\mathrm{col}\{e_1-e_2,e_1+e_2-2e_5,e_1-e_2+6e_5-6e_8\},\\
\Gamma_2=\mathrm{col}\{e_2-e_3,e_2+e_3-2e_6,e_2-e_3+6e_6-6e_9\},\\
\Gamma_3=\mathrm{col}\{e_3-e_4,e_3+e_4-2e_7,e_3-e_4+6e_7-6e_{10}\},\\
\Gamma_4=\mathrm{col}\{e_2-e_5,e_2-4e_5+3e_8\},
\Gamma_5=\mathrm{col}\{e_3-e_6,e_4-e_7\},\\
\Gamma_6=\mathrm{col}\{e_3-4e_6+3e_9,e_4-4e_7+3e_{10}\},\\
\Pi_0(h)=\mathtt{He}(\Omega(h)^TP(\Omega_2-\Omega_1))+\Omega_1^TP\Omega_1-\Omega_2^TP\Omega_2,\\
 \Pi_1=e_1^TQ_1e_1-e_2^TQ_1e_2+e_2^TQ_2e_2-e_4^TQ_2e_4,\\
\Pi_2=\mathcal{A}^T[h_1^2R_1+h_{12}^2R_2+\gamma(h_1-1)S_1+\gamma(h_{12}-1)S_2]\mathcal{A},\\
\Pi_3=\Gamma_1^T\tilde{R}_1(h_1)\Gamma_1,\;\;
\Pi_4=\begin{bmatrix}\Gamma_2\\\Gamma_3\end{bmatrix}^T
\begin{bmatrix}\tilde{R}_2&X\\ X^T &\tilde{R}_2\end{bmatrix}
\begin{bmatrix}\Gamma_2\\\Gamma_3\end{bmatrix},\\
\Pi_5=\frac{2(h_1+1)}{h_1}\Gamma_4^T\hat{S}_1(h_1)\Gamma_4,
\Pi_6=2\Gamma_5^T\hat{S}_2\Gamma_5+4\Gamma_6^T\hat{S}_2\Gamma_6,\\
\tilde{R}_1(h_1)=\mathrm{diag}\{R_1,3c_1(h_1)R_1,5c_2(h_1)R_1\},\\
\tilde{R}_2=\mathrm{diag}\{R_2,3R_2,5R_2\},\\
\hat{S}_1(h_1)=\mathrm{diag}\{S_1,2c_3(h_1)S_1\},  \hat{S}_2=\mathrm{diag}\{S_2,S_2\},
\end{IEEEeqnarray*}
where $c_i(h_1)=1$, $i=1,2,3$, if $h_1=1$ and $c_1(h_1)=(h_1+1)/(h_1-1)$,
$c_2(h_1)=(h_1+1)(h_1+2)^2/((h_1-1)(h_1^2+11))$,
$c_3(h_1)=(h_1+2)/(h_1-1)$ if $h_1>1$.

The following reciprocally convex combination inequality \cite{PKJ} will be used in the proof of
our results.

\begin{lemma}\label{lm4.1}
For given matrices $R_1\in\Bbb{S}^+_n$,
$R_2\in\Bbb{S}^+_m$, any matrix $X\in\mathbb{R}^{n\times m}$ satisfying
$\begin{bmatrix} R_1&X\\ * &R_2\end{bmatrix}\geq 0$, the inequality
\[
\begin{bmatrix} \frac{1}{\alpha}R_1 & 0\\ 0 &\frac{1}{1-\alpha}R_2\end{bmatrix} \geq
\begin{bmatrix} R_1&X\\  * & R_2\end{bmatrix}
\]
holds for all $\alpha\in (0,1)$.
\end{lemma}

\begin{proof}
An elementary proof is derived from the fact that for any positive scalars $a,b$, the
inequality
\[
\frac{a}{\alpha}+\frac{b}{1-\alpha}\geq (\sqrt{a}+\sqrt{b})^2\geq a+b+2c
\]
holds for all $\alpha\in(0,1)$ and scalar $c$ subject to $ab\geq c^2$.
\end{proof}

\begin{theorem}\label{thm4.1}
Assume that there exist symmetric positive definite matrices
$P\in\Bbb{S}^+_{4n}$, $Q_i, R_i,S_i\in\Bbb{S}^+_n$, $i=1,2$,
and a matrix $X\in\Bbb{R}^{3n\times 3n}$
such that the following LMIs hold for $h\in\{h_1,h_2\}$
\begin{IEEEeqnarray}{l}
\begin{bmatrix} \tilde{R}_2 & X\\ *& \tilde{R}_2\end{bmatrix}\geq 0,\label{e4.2}\\
\Pi(h)=\Pi_0(h)+\sum_{i=1}^2\Pi_i-\sum_{j=3}^6\Pi_j<0.\label{e4.3}
\end{IEEEeqnarray}
Then system \eqref{e4.1} is asymptotically stable for any time-varying delay $h(k)\in[h_1,h_2]$.
\end{theorem}

\begin{proof} Consider the following LKF
\begin{align*}
V(x^{[k]})=&\tilde{x}^T(k)P\tilde{x}(k)+\sum_{s=k-h_1}^{k-1}x^T(s)Q_1x(s)\\
&+\sum_{s=k-h_2}^{k-h_1-1}x^T(s)Q_2x(s)+h_1\sum_{s=-h_1}^{-1}\sum_{i=k+s}^{k-1}\Delta{x}^T(i)R_1\Delta{x}(i)\\
&+h_{12}\sum_{s=-h_2}^{-h_1-1}\sum_{i=k+s}^{k-1}\Delta{x}^T(i)R_2\Delta{x}(i)\\
&+\sum_{s=-h_1}^{-1}\sum_{i=-h_1}^s\sum_{j=k+i}^{k-1}\Delta{x}^T(j)S_1\Delta{x}(j)\\
&+\sum_{s=-h_2}^{-h_1-1}\sum_{i=-h_2}^s\sum_{j=k+i}^{k-1}\Delta{x}^T(j)S_2\Delta{x}(j),
\end{align*}
where $\tilde{x}(k)=\mathrm{col}\big\{x(k),\sum\limits_{s=k-h_1}^{k-1}x(s),
\sum\limits_{s=k-h_2}^{k-h_1-1}x(s),\sum\limits_{s=-h_1}^{-1}\sum\limits_{i=k+s}^{k-1}x(i)\big\}$
and $x^{[k]}$ denotes the segment $\{x(k): k\in\Bbb{Z}[-h_2,0]\}$.

The previous functional is positive definite due to
the assumptions of Theorem \ref{thm4.1}.
Now, we employ our newly derived inequalities in Lemma \ref{lm3.1} and
Lemma \ref{lm3.2} in bounding $\Delta V(x^{[k]})$. 
Note at first that $\tilde{x}(k+1)=(\Omega(h)-\Omega_1)\zeta_0(k)$,
$\tilde{x}(k)=(\Omega(h)-\Omega_2)\zeta_0(k)$ and 
$\Delta(\tilde{x}^T(k)P\tilde{x}(k))=(\tilde{x}(k)+\tilde{x}(k+1))^TP\Delta\tilde{x}(k)$.
Then we have
\begin{equation}\label{e4.4}
\begin{aligned}
&\Delta V(x^{[k]})=\zeta_0^T(k)\left(\Pi_0(h)+\Pi_1+\Pi_2\right)\zeta_0(k)\\
&-h_1\sum_{s=k-h_1}^{k-1}\Delta^T{x}(s)R_1\Delta{x}(s)
-h_{12}\sum_{s=k-h_2}^{k-h_1-1}\Delta^T{x}(s)R_2\Delta{x}(s)\\
&-\sum_{s=-h_1}^{-1}\sum_{i=k-h_1}^{k+s}\Delta^T{x}(i)S_1\Delta{x}(i)
-\sum_{s=-h_2}^{-h_1-1}\sum_{i=k-h_2}^{k+s}\Delta^T{x}(i)S_2\Delta{x}(i).
\end{aligned}
\end{equation}

Note that, the following equality
\begin{equation}\label{e4.5}
\sum_{s=a}^{b}\sum_{i=a}^sv(i)=(b-a+2)\sum_{s=a}^bv(s)-\sum_{s=a}^b\sum_{i=s}^bv(i)
\end{equation}
holds for any sequence $v:\Bbb{Z}[a,b]\to\Bbb{R}^n$.
Using \eqref{e4.5} in presenting $\upsilon_3$ defined in Lemma \ref{lm3.1},
we have
\begin{equation}\label{e4.6}
-h_1\sum_{s=k-h_1}^{k-1}\Delta^T{x}(s)R_1\Delta{x}(s)
\leq -\zeta_0^T(k)\Gamma_1^T\tilde{R}_1(h_1)\Gamma_1\zeta_0(k).
\end{equation}

Similarly, the second summation term of \eqref{e4.4} can be bounded
by \eqref{e3.16} and Lemma \ref{lm4.1} as follows
\begin{align*}
-h_{12}\sum_{s=k-h_2}^{k-h_1-1}&\Delta^T{x}(s)R_2\Delta{x}(s)
\leq -\frac{h_{12}}{h-h_1}\zeta_0^T(k)\Gamma_2^T\tilde{R}_2\Gamma_2\zeta_0(k)\\
&-\frac{h_{12}}{h_2-h}\zeta_0^T(k)\Gamma_3\tilde{R}_2)\Gamma_3\zeta_0(k)\\
&=-\zeta_0^T(k)\begin{bmatrix}\Gamma_2\\\Gamma_3\end{bmatrix}^T
\begin{bmatrix}\frac{h_{12}}{h-h_1}\tilde{R}_2 &0\\ 0&\frac{h_{12}}{h_2-h}\tilde{R}_2\end{bmatrix}
\begin{bmatrix}\Gamma_2\\\Gamma_3\end{bmatrix}\zeta_0(k)\\
&\leq -\zeta_0^T(k)\begin{bmatrix}\Gamma_2\\\Gamma_3\end{bmatrix}^T
\begin{bmatrix}\tilde{R}_2 &X\\ *&\tilde{R}_2\end{bmatrix}
\begin{bmatrix}\Gamma_2\\\Gamma_3\end{bmatrix}\zeta_0(k).
\end{align*}
Note that, when $h=h_1$ and $h=h_2$ then $\Gamma_2\zeta_0(k)=0$ and
$\Gamma_3\zeta_0(k)=0$, respectively, and thus
the last inequality is still valid. Therefore
\begin{equation}\label{e4.7}
-h_{12}\sum_{s=k-h_2}^{k-h_1-1}\Delta^T{x}(s)R_2\Delta{x}(s)\leq -\zeta_0^T(k)\Pi_4\zeta_0(k).
\end{equation}

By Lemma \ref{lm3.2} we have
\begin{equation}\label{e4.8}
-\sum_{s=-h_1}^{-1}\sum_{i=k-h_1}^{k+s}\Delta{x}^T(i)S_1\Delta{x}(i)
\leq -\zeta_0^T(k)\Pi_5\zeta_0(k).
\end{equation}
Now, we employ \eqref{e3.17} to bound the last term in \eqref{e4.4}.
To do this, note at first that
\begin{align*}
\sum_{s=-h_2}^{-h_1-1}\sum_{i=k-h_2}^{k+s}\Delta{x}^T(i)S_2\Delta{x}(i)
&\geq \sum_{s=-h}^{-h_1-1}\sum_{i=k-h}^{k+s}\Delta{x}^T(i)S_2\Delta{x}(i)\\
&+\sum_{s=-h_2}^{-h-1}\sum_{i=k-h_2}^{k+s}\Delta{x}^T(i)S_2\Delta{x}(i).
\end{align*}
Then, by applying \eqref{e3.17} and rearranging the obtained results we get
\begin{equation}\label{e4.9}
-\sum_{s=-h_2}^{-h_1-1}\sum_{i=k-h_2}^{k+s}\Delta{x}^T(i)S_2\Delta{x}(i)
\leq -\zeta_0^T(k)\Pi_6\zeta_0(k).
\end{equation}

It follows from \eqref{e4.4}-\eqref{e4.9} that
\begin{equation}\label{e4.10}
\Delta V(x^{[k]})\leq \zeta_0^T(k)\Pi(h)\zeta_0(k).
\end{equation}
The matrix $\Pi(h)$ is an affine function $h$, and thus,
$\Pi(h)<0$ for all $h\in[h_1,h_2]$ if and only if $\Pi(h_1)<0$ and $\Pi(h_2)<0$. 
Therefore, if \eqref{e4.3} holds for $h=h_1$ and $h=h_2$ then, from \eqref{e4.10},
$\Delta V(x^{[k]})$ is negative definite which ensures the asymptotic stability of system \eqref{e4.1}.
The proof is completed.
\end{proof}

\begin{remark}\label{rm4.1}
In Theorem \ref{thm4.1}, a full $3n\times 3n$ matrix $X$ is used in the reciprocally inequality
to improve the upper bound of delay. However, to reduce the number of decision variables,
we can use the matrix $X$ of the form $X=\mathrm{diag}\{X_1,X_2,X_3\}$, where $X_i\in\Bbb{R}^{n\times n}$,
$i=1,2,3$.
\end{remark}

\subsection{Examples}\label{sec:5}

\begin{example}
Consider system \eqref{e4.1} with the matrices taken from the literature
\[A=\begin{bmatrix} 0.8 & 0.0\\ 0.05 & 0.9\end{bmatrix},\quad
A_1=\begin{bmatrix} -0.1 & 0.0\\ -0.2 & -0.1\end{bmatrix}.
\]

The obtained results and comparison to most recent results in the literature
are given in Table I and Table II. It is worth noting that, thanks to our new
summation inequalities proposed in Lemma \ref{lm3.1} and Lemma \ref{lm3.2},
Theorem \ref{thm4.1} clearly delivers significantly better results than the
existing methods in the literature. Especially, our method requires less decision variables
than the proposed conditions in \cite{NPH,Kwon1,Kim}
while leading to much better results.
\end{example}

  \begin{table*}[!htb]\label{tb1}
  \caption{Upper bounds of $h_2$ for various $h_1$ in Example 1}
  \begin{tabular*}{\textwidth}{l @{\extracolsep{\fill}} lllllllll}\hline
  $h_1$& 2 & 4 & 6 & 10 & 15&20&25&30& NoDv\\ \hline
\cite{SH} (Proposition 1) &17&17&18&20&23&27&31&35& $8n^2+3n$\\
\cite{Peng} (Theorem 1) &18&18&19&20&23&26&30&35& $3.5n^2+3.5n$\\
\cite{ZPZ} (Theorem 3.1, $l=4$) &20&21&21&22&24&27&29&34& $9.5n^2+5.5n$\\
\cite{FLY} (Theorem 4, $l=4$) &21&21&21&22&24&27&31&35&$9.5n^2+5.5n$\\
\cite{Kwon1} (Theorem 2)&22&22&22&23&25&28&32&36&$27n^2+9n$\\
\cite{Kim} (Theorem 2) &22&22&22&23&25&28&32&36&$23n^2+7n$\\
   \cite{NPH} (Theorem 1) &22&22&22&23&26&29&32&36&$19n^2+5n$\\
Remark \ref{rm4.1} &24&26&27&30&32&33&35&39&$14n^2+5n$\\
   Theorem \ref{thm4.1} &26&27&28&31&34&35&36&39 &$20n^2+5n$\\
   \hline
   \end{tabular*}
   {NoDv: Number of Decision variable}
   \end{table*}

  \begin{table}[!htb]\label{tb2}
  \caption{Upper bounds of $h_2$ for various $h_1$ in Example 1}
  \begin{tabular*}{\textwidth}{l @{\extracolsep{\fill}} lllllll}\hline
  $h_1$& 1 & 3 & 5 & 7 & 11&13& NoDv\\ \hline
\cite{ZXZ} &12&13&14&15&17&19& $9n^2+3n$\\
\cite{HWLX}&17&17&17&18&20&22& $13n^2+5n$\\
\cite{Kao} &17&18&19&21&25&25&$\mathrm{Dv}_*$\\ 
 \cite{SGF1} &20&21&21&22&23&24&$10.5n^2+3.5n$\\
Thm. \ref{thm4.1} &26&27&28&29&32&33&$14n^2+5n$\\     \hline
   \end{tabular*}
$\mathrm{Dv}_*=(h_2+1)^2n^2/2+(h_2+2)n/2$
\end{table}

\begin{example}
Let us now consider a practical satellite control system \cite{GMCL}.
The dynamic equations are as follows
\begin{equation}\label{e4.11}
\begin{aligned}
&J_1\ddot\theta_1(t)+f(\dot\theta_1(t)-\dot\theta_2(t))
+k(\theta_1(t)-\theta_2(t))=u(t),\\
&J_2\ddot\theta_2(t)+f(\dot\theta_1(t)-\dot\theta_2(t))
+k(\theta_1(t)-\theta_2(t))=0,
\end{aligned}
\end{equation}
where $J_i$, $i=1,2$, are the moments of inertia of the two bodies,
$f$ is a viscous damping, $k$ is a torque constant, $\theta_i(t)$ are the
yaw angles for the two bodies and $u(t)$ is a control input.
The following parameters are borrowed from \cite{GMCL}:
$J_1=J_2=1$, $k=0.09, f=0.04$. Let $x_i(t)=\theta_i(t)$,
$x_{i+2}(t)=\dot\theta_i(t)$, $i=1,2$. By choosing a sampling time $T=10$ ms,
system \eqref{e4.11} can be transformed to the following discrete-time system
\cite{Kwon1}
\begin{equation}\label{e4.12}
x(k+1)=Ax(k)+Bu(k)
\end{equation}
where
\[
A=\begin{bmatrix}1&0&0.01&0\\0&1&0&0.01\\
-0.009&0.009&0.9996&0.0004\\
0.009&-0.009&0.0004&0.9996\end{bmatrix},
B=\begin{bmatrix}0\\0\\0.01\\ 0\end{bmatrix}.
\]

A delayed state feedback controller is designed in the form $u(k)=Kx(k-h(k))$,
where $h(k)$ is time-varying delay belonging to the interval $[h_1,h_2]$.
For $h_1=1$, it was found that with the controller gain
$K=\begin{bmatrix}0.1284&-0.1380&-0.3049 &0.0522\end{bmatrix}$,
Theorems 1 and 2 in \cite{Kwon1} give the upper bounds of $h_2$ as 
129 and 135, respectively, which are larger than 98 delivered by the results of \cite{ZXZ}.
We apply Theorem \ref{thm4.1} for $A_d=BK$, it is found that the closed-loop system
remains asymptotically stable for the time-varying delay $h(k)\in[1,170]$ which
shows a clear reduction of the conservatism. To demonstrate the effectiveness of the obtained
result, a simulation with $h(k)=1+169|\sin(k\pi/2)|$ and initial condition
$[2\; -1\;0.2\; -0.5]^T$ is presented in Fig. 1. It can be seen that the state trajectory converges
to zero as shown by our theoretical result.

\begin{figure}[!ht]
\centering
\includegraphics[width=0.68\textwidth]{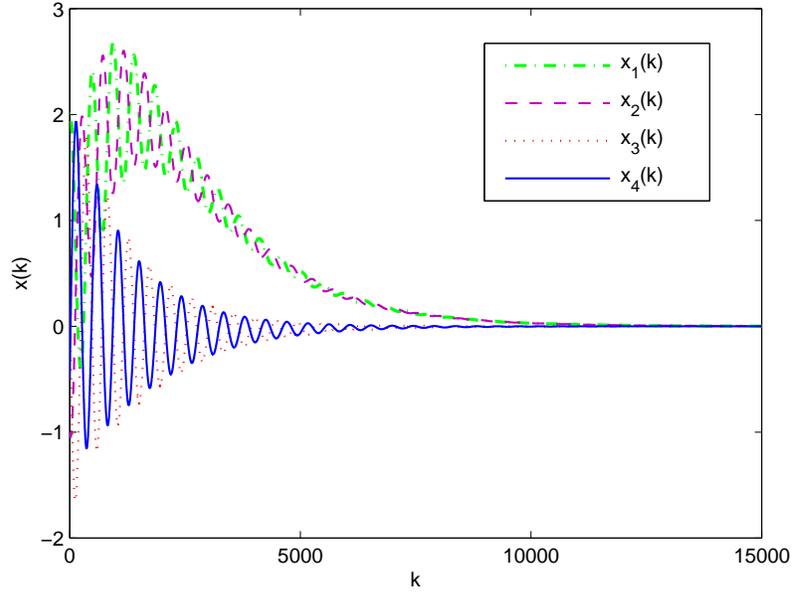}
\caption{Responses of the satellite system with $h(k)=1+169|\sin(k\pi/2)|$}\label{fig:1}
\end{figure}
\end{example}

\section{Conclusion}
In this paper, new summation inequalities in single and double form have been proposed.
By employing the newly derived inequalities, improved stability conditions
have been derived for a class of discrete-time systems with
time-varying delay. Provided examples and comparisons to
the most recent results found in the literature 
show the potential and a large improvement
on the stability conditions deliver by the approach proposed
in this paper.


\end{document}